\documentclass[11pt, leqno]{amsart}
\usepackage[utf8]{inputenc}
\usepackage{graphicx}
\usepackage[lmargin=25mm,rmargin=25mm,tmargin=33mm,bmargin=30mm]{geometry}
\usepackage[english]{babel}

\usepackage{protosem}
\usepackage{xcolor}

\usepackage{amsmath,amsfonts}
\usepackage{amssymb}
\usepackage{amsthm}
\usepackage{mathtools}
\usepackage{csquotes}
\usepackage{mathrsfs}
\usepackage{textcomp}

\makeatletter
\@namedef{subjclassname@2020}{%
  \textup{2020} Mathematics Subject Classification}
\makeatother

\newtheorem{theorem}{Theorem}[section]
\newtheorem{corollary}[theorem]{Corollary}
\newtheorem{lemma}[theorem]{Lemma} 
 
\newtheorem{remark}[theorem]{Remark}
\numberwithin{equation}{section}
\theoremstyle{definition}
\newtheorem{definition}[theorem]{Definition}

\newcounter{example}

\newcommand{\C}{\ensuremath{{\mathbb C}}}

\newcommand{\Q}{\ensuremath{{\mathbb Q}}}
\newcommand{\Z}{\ensuremath{{\mathbb Z}}}

\author{ NICHOLAS AIDOO}

\address{Department of Mathematical Sciences, Norwegian University of Science and Technology, Trondheim, Norway}

\email{nicholas.aidoo@ntnu.no}



\title{CATLIN'S BOUNDARY SYSTEMS FOR SUMS OF SQUARES DOMAINS}
\date{}

\begin{document}

\maketitle

\noindent ABSTRACT: {\normalsize  For any given sum of squares domain in $\C^n,$ we reduce the complexity in Catlin's multitype techniques by giving a complete normalization of the geometry. Using this normalization result, we present a more elementary proof of the equality of the Catlin multitype and the commutator multitype for such domains when both invariants are finite. Finally, we reformulate algebraically Catlin's machinery for the commutator multitype computation at the origin for any given sum of squares domain in $\C^n$.}

\vspace{1cm}

\section{Introduction} \label{s1}

The multitype introduced by Catlin in \cite{catlinbdry} is one of the significant CR-invariants studied in most problems related to the boundary geometry of pseudoconvex domains and its applications. The definition of the multitype gives no hint how to compute it and so Catlin addresses this issue in \cite{catlinbdry} by introducing another CR-invariant known as the commutator multitype, which is computed by differentiating the Levi form along certain lists of vector fields generated from a geometric object known as \textit{a boundary system}.  He further established that the two boundary invariants are equal for any given pseudoconvex domain. The multitype is also related to the boundary invariant, which was initially studied by Kohn in \cite{kboundary} for hypersurfaces in $\C^2$ and later by Bloom and Graham \cite{bloomgraham} for hypersurfaces in $\C^n.$

The goal of this paper is to give a simple and elementary proof of the equality of the Catlin multitype  and the commutator multitype at a boundary point of any sum of squares domain in $\C^n$ given that both boundary invariants are finite. Our first main theorem is thus formulated below. 

\begin{theorem} \label{equalm}
	Let $0 \in M$ be a smooth real hypersurface in $\C^n$ with defining function $$ r = \mathrm{Re}\, z_1+ \sum_{j=1}^N |f_j(z_2, \dots, z_n)|^2,  $$ where $f_1, \ldots, f_N$ are holomorphic functions near $\mathrm{0}.$ Suppose that $\mathfrak{B}_n(0)$ is a boundary system defined near the origin. Denote by $\mathfrak{C}(0)$ and $\mathfrak{M}(0),$ the commutator multitype and the multitype at the origin respectively and suppose that they are both finite.
	
	Then 
	\[
	\mathfrak{C}(0) = \mathfrak{M}(0).
	\]
	
\end{theorem}

Even with the introduction of the commutator multitype, the multitype is still quite difficult to compute. Since the multitype has many interesting properties, we seek to obtain new ideas and methods that can reduce the complexities in its computation and also give a simpler geometric or algebraic characterization. For instance, obtaining an in-depth geometric or algebraic insight into the multitype techniques, polynomial models and the boundary systems developed by Catlin in \cite{catlinbdry} is relevant to the study of global regularity and subelliptic estimates of the $\bar \partial$-Neumann problem on pseudoconvex domains established in \cite{C84, C87}. Some works that can be directly or indirectly related to the problem of finding new techniques to reduce the difficulty in the computation of the multitype are given below.

McNeal \cite{mcneal} proved the equality of  the order of contact of the boundary with complex analytic varieties (the D'Angelo type, see \cite{opendangelo}) and the order of contact with complex lines, for convex domains of finite type in $\C^n.$ Boaz and Straube \cite{bs} gave a direct geometric proof of the same equality for convex domains as proved by McNeal and extended their results to a wider class of domains. Yu \cite{jiyeyu} studied the multitype for convex domains and obtained a simple geometric characterization of the multitype for convex domains in terms of the orders of contact of complex lines with the boundary of the domain. Kol\' a\v r \cite{martinIMRN} introduced a constructive algorithm for the computation of the multitype when all its entries are finite using weighted homogeneous polynomial transformations. Zaitsev in \cite{DZ19} introduced new invariant ideal sheaves containing functions arising in Catlin's boundary systems and gave a more direct computation of the kernels of some newly introduced invariant tensors that simplify the iterative construction of such boundary systems. Basyrov, Nicoara, and Zaitsev \cite{bdn} presented a technique that further reduced the complexity of the multitype computation by establishing that the polynomial model of every pseudoconvex domain of finite Catlin multitype contains a simple sum of squares of monomials, the so-called \textit{balanced sum of squares}. Recently, the author in \cite{Aidoo2022} applied the Kol\' a\v r algorithm for the multitype computation introduced in \cite{martinIMRN} to a class of manifolds defined by sums of squares of holomorphic functions, also referred to as \textit{special domains} by Kohn in \cite{kboundary} and proved two results. The first result proved that the model of a sum of squares domain is likewise a sum of squares and the second result established that the multitype is an invariant of the ideal of holomorphic functions defining the domain. Also, in the same paper, the author gave an algebraic reformulation of the Kol\' a\v r algorithm for the multitype computation and further gave a method that explicitly constructed the weighted homogeneous polynomial transformations needed in the Kol\' a\v r algorithm.

An interesting discovery in \cite{bdn} is the obstruction to obtaining a complete normalization of any Catlin boundary system for a general pseudoconvex domain via any holomorphic change of coordinates. This failure to obtain a complete normalization is attributed to the occurrences of the so-called \textit{torsion} in the boundary. Obviously, any strongly pseudoconvex domain will have no torsion in the boundary since there always exists a maximal Levi-nondegenerate subbundle that gives a nonsingular Levi matrix at the origin. Thus, a complete normalization of the associated boundary system can always be achieved for such domains. It is therefore of crucial interest to find out the classes of pseudoconvex domains for which there exist a complete normalization of the Catlin boundary systems. In this paper, we ask the following natural questions for any given sum of squares domain in $\C^n:$  

\begin{itemize}
	
	\item[(i).]  Can torsion occur in the boundary?
	
	\item[(ii).] Can the associated Catlin boundary system of such a domain be completely normalized?
\end{itemize} 

\noindent We will show that a negative answer to the first question answers the last question in the affirmative, which brings us to our second main theorem:

\begin{theorem} \label{norm}
	Let $0 \in M$ be a smooth real hypersurface in $\C^n$ with defining function $$ r = \mathrm{Re}\, z_1+ \sum_{j=1}^N |f_j(z_2, \dots, z_n)|^2,  $$ where $f_1, \ldots, f_N$ are holomorphic functions near $\mathrm{0}.$ Let the Levi rank at $\mathrm{0}$ be $q$ and of Catlin multitype at $\mathrm{0}$  $$\Lambda = (1, \underbrace{2, \dots, 2}_{q}, \lambda_{q+2}, \dots, \lambda_n),$$ where $2<\lambda_{q+2} \leq \cdots \leq \lambda_n < +\infty.$ Then for any boundary system at $\mathrm{0},$ 
	\[
	\mathfrak{B}_{n}(0) = \{ r_1, r_{q+2}, \dots, r_n; \ L_2, \dots, L_n \},
	\]
	there exists a holomorphic change of coordinates at $\mathrm{0}$ that preserves the multitype such that $\mathfrak{B}_n(0)$ becomes
	\[
	\tilde{\mathfrak{B}}_n(0) = \{ \tilde{r}_1, \tilde{r}_{q+2}, \dots, \tilde{r}_n;  \tilde{L}_2, \dots, \tilde{L}_n \}
	\] 
	and satisfies the normalization
	
	\begin{equation}
	\begin{split}
	\tilde r_k  &= \mathrm{Re}\, z_k + o(\lambda_k^{-1} ), \quad q +2 \leq k \leq n\\
	\tilde{L}_k &= \partial_{z_k}+ o(\lambda_k^{-1} ), \quad 2 \leq k \leq n
	\end{split}
	\end{equation}
	where the partial derivatives are counted with weight $-\lambda_k^{-1}.$
	
\end{theorem}

Results from Theorem \ref{norm} will be crucial in our proof of Theorem \ref{equalm}. This work also constitute a part of a broader research motivated by a question posed by J. P. D'Angelo, namely how does the stratification of the boundary of any sum of squares domain by the multitype level sets look like? As a contribution towards this research program, the author in \cite{Aidoo2022} relies on the Kol\'a\v r algorithm for the multitype computation in \cite{martinIMRN} to give an ideal reformulation of the multitype computation when all its entries are finite. Hence, another motivation in this paper is to obtain an ideal reformulation of Catlin's machinery for the commutator multitype computation. This transition allows one to compute the commutator multitype from the corresponding ideal of holomorphic functions defining the domain. The reformulation also has the advantage of reducing significantly the level of difficulty in computing the commutator multitype of any sum of squares domain in $\C^n.$

The paper is organized as follows: Section 2 provides preliminary notions and definitions pertinent to the discussion in the subsequent sections. In section 3, a result which proves that the commutator multitype is crossterm invariant is given. This is an important result needed for restating Catlin's machinery in terms of the ideals of holomorphic functions. Section 4 introduces the notion of derivatives of the ideal of holomorphic functions for a sum of squares domain under the assumption that its corresponding commutator multitype can be computed. Section 5 presents an important lemma that establishes the absence of torsion in the boundary of any sum of squares domain. The proofs of the main theorems are given at the end of this section. In Section 6, an ideal restatement of Catlin's machinery for the commutator multitype computation is presented.

\section{Definitions and Notation}

We give some basic definitions related to the multitype introduced by Catlin in \cite{catlinbdry}. We shall also present some other related definitions as given in \cite{Aidoo2022}.

\begin{definition} \label{sosdom}
	Let $D \subset \C^n$ be a smooth bounded domain with boundary-defining function $r(z)$ defined by
	\begin{equation} \label{sos} 
	r(z) = 2 \text{Re}(z_{1}) + \sum_{j=1}^{N} |f_j(z_2, \dots, z_{n})|^2,
	\end{equation}
	where $f_j(z_2, \dots, z_{n})$ for all $j, \ 1 \leq j \leq N$ are holomorphic functions vanishing at the origin in $\C^{n}$. The domain $D \subset \C^n$ is called a \textit{sum of squares domain}. It is also referred as a \textit{special domain} (see \cite{kboundary}).

	Denote by $M \subset \C^{n}$ the hypersurface defined by $\{ z \in \C^{n} \ | \ r(z) = 0 \}.$
\end{definition}

\begin{definition} \label{modhyp}
	The \textit{model hypersurface} associated to $M$ at the origin is defined as
	\begin{equation} \label{mod}
	M_0 = \{z \in \C^{n} : \ p(z, \bar z) = 0 \}, 
	\end{equation}
	the zero locus of the homogeneous polynomial $p(z, \bar z)$ consisting of all monomials from the Taylor expansion of the defining function that have weight 1 with respect to the multitype weight. We call  $p(z, \bar z)$ the \textit{model polynomial}. 
\end{definition}

\begin{definition}
	A weight $\Lambda = (\lambda_1, \dots, \lambda_n)$ is an n-tuple of rational numbers with $1 \leq \lambda_j \leq +\infty$ satisfying:
	\begin{itemize}
		\item[i.] $\lambda_j \leq \lambda_{j+1}$ for $1 \leq j \leq n-1;$ 
		\item[ii.] For each $t$, either $\lambda_t = +\infty$ or there exists a sequence of nonnegative integers $a_1, \dots, a_t$ satisfying $a_j > 0$ for all $1 \leq j \leq t$ and
		\[\sum_{j=1}^t \frac{a_j}{\lambda_j} = 1.\]
	\end{itemize}
	
\end{definition}

\medskip

We shall denote by $\Gamma_n$ the set of all weights $\Lambda$ as defined above.
We define a lexicographic ordering on $\Gamma_n,$ i.e. for any given $\Lambda_1, \ \Lambda_2 \in \Gamma_n$ such that $\Lambda_1 = (\lambda_1, \dots, \lambda_n)$ and $\Lambda_2 = (\lambda_1^{'}, \dots, \lambda_n^{'})$, then $\Lambda_1 > \Lambda_2$ if for some $t$, $\lambda_j = \lambda_j^{'}$ for $j <t$ and $\lambda_t> \lambda_t^{'}$.

\begin{definition} \label{distg}
	Let $\Omega \in \C^n$ be a smooth domain with defining function $r$ and let $z_0$ be a point on the boundary $b\Omega.$ We say that the weight $\Lambda = (\lambda_1, \dots, \lambda_n)$ is \textit{distinguished} at $z_0$ if there exists holomorphic coordinates $(z_1, \dots, z_n)$ about $z_0$ such that 
	\begin{itemize}
		\item[i.] $z_0$ is mapped to the origin;
		
		\item[ii.] If $\sum_{j=1}^n \frac{\alpha_j + \bar{\beta}_j}{\lambda_j} <1,$ then $D^{\alpha} \bar{D}^{\bar{\beta}} r(0) =0,$ where $D^{\alpha} = \frac{\partial^{|\alpha|}}{\partial z_1^{\alpha_1} \cdots \partial z_n^{\alpha_n}}$ and  $\bar{D}^{\bar{\beta}} = \frac{{\partial}^{\bar{|\beta|}}}{{\partial} \bar{z}_1^{\bar{\beta}_1} \cdots {\partial} \bar{z}_n^{\bar{\beta}_n}}.$ 
	\end{itemize}
\end{definition}

\begin{remark} \label{distgrk}
	The property (ii) of Definition \ref{distg} can be interpreted as follows: there exist holomorphic coordinates about $z_0$ such that the boundary-defining function $r$ in the new coordinates is of the form 
	\begin{equation} \label{distgeq}
	r = \mathrm{Re}\, z_1+p(z_2, \dots, z_n, \bar{z}_2, \dots, \bar z_{n}) + o_{\Lambda}(1),
	\end{equation}
	where $p(z_2, \dots, z_n, \bar{z}_2, \dots, \bar z_{n})$ is a $\Lambda$-homogeneous polynomial of weighted degree 1 without pluriharmonic terms and $o_{\Lambda}(1)$ denotes a smooth function whose derivatives of weighted order less than or equal to 1 vanish at zero.
\end{remark}

We shall denote the set of all distinguished weights at $z_0$ by $\tilde{\Gamma}(z_0).$

\begin{definition} \label{defa}
	The \textit{multitype} $\mathfrak{M}(z_0)$ at $z_0$ is defined to be the smallest weight in lexicographic sense $\mathfrak{M}(z_0) = (m_1, \dots, m_n)$ such that $\mathfrak{M}(z_0) \geq \Lambda$ for every distinguished weight $\Lambda \in \tilde{\Gamma}(z_0).$
\end{definition}
We call the multitype $\mathfrak{M}(z_0)$ at $z_0$ finite, if the last entry $m_n< +\infty$. We also recall that each $m_i$ of the multitype satisfies the inequality $$ m_{n-q+1} \leq \Delta_{q}(M,z_0),$$
where $\Delta_{q}(M,z_0)$ is the D'Angelo q-type as defined in \cite{opendangelo, dangelo}. Thus, the last entry $m_n$ is bounded above by the D'Angelo 1-type $\Delta_{1}(M,z_0).$

\medspace 

\noindent Now, for every sum of squares domain $D$ as defined in Definition \ref{sosdom}, we associate the ideal of holomorphic functions $\left< \tilde{f}\right> = \left< z_1, f_1, \dots, f_N \right>.$ Since the multitype gives a refined measure of the vanishing order of the boundary-defining function $r,$ it is easy to see that the order of vanishing of $r$ in the normal direction is 1. Hence, we assign the weight 1 to the first entry of the multitype, i.e. $m_1 =1.$ We are most interested in the vanishing order of $r$ in the direction of the complex tangential variables $z_2, \dots, z_n.$ For the rest of the paper, unless stated otherwise, for any given sum of squares domain $D$ we will rather associate the ideal of holomorphic functions $\left< f\right> = \left< f_1, \dots, f_N \right>.$ Now, consider the following definitions as given in \cite{Aidoo2022}.

\begin{definition} \label{gradid}
	Let $f \in \C[z_2, \dots, z_n],$ be a polynomial in the variables $z_2, \dots, z_n$ with coefficients in $\C.$ We define the \textit{gradient ideal of f} as the ideal generated by the partial derivatives of $f:$
	\begin{equation} \label{gi1}
	\mathcal{I}_{grad}(f) = \left <\nabla f \right > = \left <\frac{\partial f}{\partial z_2}, \cdots, \frac{\partial f}{\partial z_n} \right >.
	\end{equation}
\end{definition}

\begin{definition} \label{gradid2}
	Given the ideal $\left<f\right> = \left<f_1, \dots, f_N\right> \subset \C[z_2, \dots, z_n],$ we define the \textit{Jacobian module of $f$} as 
	\begin{equation} \label{gi2}
	\mathfrak{J}_{\left<f\right>} = \left[\frac{\partial f}{\partial z_2}, \cdots, \frac{\partial f}{\partial z_n} \right],
	\end{equation}
	where each $\frac{\partial f}{\partial z_j}$ is a vector. $\mathfrak{J}_{\left<f\right>}$ is a module over the polynomial ring $\C[z_1, \dots, z_n].$
\end{definition}
We associate the complex Jacobian matrix $\text{J}(f)$ given by
\begin{equation} \label{ji}
\text{J}(f) = \begin{pmatrix}
\frac{\partial f_2}{\partial z_2}& \cdots & \frac{\partial f_N}{\partial z_2}\\
\vdots & \ddots & \vdots \\
\frac{\partial f_2}{\partial z_n} & \cdots & \frac{\partial f_N}{\partial z_n}
\end{pmatrix}
\end{equation}
to every Jacobian module $\mathfrak{J}_{\left<f\right>}.$ For each gradient ideal $\mathcal{I}_{grad}(f_i) = \left <\frac{\partial f_i}{\partial z_2}, \cdots, \frac{\partial f_i}{\partial z_n} \right >$ of the generator $f_i \in \C[z_1, \dots, z_n]$ of $\left<f\right>,$ we associate the $i$-th column of $\text{J}(f)$ for $1 \leq i \leq n.$

\section{Crossterm Invariance of Catlin's Commutator Multitype}

We refer the reader to \cite{catlinbdry} and \cite{bdn} for a thorough and detailed description of Catlin's machinery for the computation of the commutator multitype and the associated boundary systems. The commutator multitype, as defined in \cite{catlinbdry}, is a list of $n$ positive rational numbers, which are obtained via the commutator properties of a system of tangential holomorphic vector fields. We shall denote the commutator multitype by $\mathfrak{C}(z)$  and let $\mathfrak{C}(z) = (c_1, \dots, c_n).$

We seek to reformulate Catlin's machinery for the commutator multitype computation for a sum of squares domain in terms of the corresponding ideals of holomorphic functions defining the domain. This restatement in terms of ideals of holomorphic functions, however, can only hold if the commutator multitype at the origin is crossterm invariant. By crossterm invariance we mean that the value of the commutator multitype is unaltered with the inclusion or omission of crossterms from the expansion of the moduli squares of all generators of such domains. More particularly, we will show that each entry of the commutator multitype can always be realized by the modulus square of some monomial from the Taylor expansion of some generator of the sum of squares domain in $\C^n.$ For the purposes of the material discussed in this section, we shall recall some definitions and notation from \cite{catlinbdry}.

Let $\Omega$ be a smoothly bounded domain and let $z_0$ be a given point in the boundary $b\Omega = \{r=0\}.$ Denote by $r$ a smooth boundary-defining function and suppose that the Levi-form of $b\Omega$ at $z_0$ has rank equal to $q.$

\begin{definition} \label{catorder}
	Let $\mathscr{L} = \{ L^1, \ldots, L^l \}$ be a list of vector fields in $T^{1,0}b\Omega$ such that $L^j$ is one of  $\{L_{c_1}, \bar{L}_{c_1} \ldots, L_{c_t}, \bar{L}_{c_t} \}$ for some $c_j \in \Z_{>0}.$ Let $l_i$ denote the total number of times $L_i$ and $\bar{L}_i$ appear in $\mathscr{L}.$ We say $\mathscr{L}$ is \textit{ordered} (in the sense of Catlin) if
	
	\[
	L^j = L_{c_t} \ \ \text{or} \ \ \bar{L}_{c_t} \ \ \text{when} \quad 1  \leq j \leq l_{c_t}, \quad \text{and}
	\]
	
	\[
	L^j = L_i \ \ \text{or} \ \ \bar{L}_i \ \ \text{when} \quad 1 + \sum_{k=i+1}^{c_t} l_k \leq j \leq 	\sum_{k=i}^{c_t} l_k.
	\]
\end{definition}

\begin{definition}
	An ordered list $\mathscr{L} = \{ L^1, \ldots, L^l \}$ is said to be $\nu$-admissible if $l_{\nu} >0$ and $\mathscr{L}$ satisfies the inequality
	\[
	\sum_{k=q+2}^{\nu-1} \frac{l_k}{c_k} <1.
	\]
\end{definition}

\noindent  Let $T^{1,0}_{\nu}$ denote the subbundle of $T^{1,0} b\Omega$ defined by
\[
T^{1,0}_{\nu} = \{ L \ | \ L(r_k)=0,  \quad \text{and} \quad \partial \bar\partial r(L, \bar L_j)=0, \ j =2 , \dots, q+1 \},
\]
where $ k=1, q+2, \dots, \nu-1$ and $q+2 \leq \nu -1 < n.$ Let $l \in \Z$ with $l \geq 3$ and denote by $\mathscr{L}$ a given list of vector fields, $\mathscr{L} = \{L^1, \ldots, L^l\}$ such that there exists a nonvanishing vector field $L_{\nu} \in T^{1,0}_{\nu}.$ 
Denote by  $\mathscr{L} \partial r$ the function
\begin{equation} \label{fxn}
\mathscr{L} \partial r(z) = L^1 \cdots L^{l-2} \partial r ([L^{l-1}, L^l])(z)
\end{equation}
for $z \in b\Omega.$ We know from Catlin's construction \cite{catlinbdry} that for every finite entry of the commutator multitype $\mathfrak{C}(z_0)$ at $z_0 \in b\Omega,$ there exists at least one $\nu$-admissible ordered list $\mathscr{L} = \{L^1, \ldots, L^l\}$ of minimal length $l \geq 3$ such that $\mathscr{L}\partial r(z_0) \neq 0$ and also satisfies the equation 
\begin{equation} \label{comm}
\sum_{k=q+2}^{\nu-1} \frac{l_k}{c_k} + \frac{l_{\nu}}{c(\mathscr{L})} =1.
\end{equation}
Here, $l_k$ is the number of times that both $L_k$ and $\bar{L}_k$ appear in the list $\mathscr{L}$ and $c(\mathscr{L})$ denotes the solution \eqref{comm}. The chosen list $\mathscr{L}$ satisfying the above properties is denoted by $\mathscr{L}_{\nu}.$ Since $\mathscr{L}_{\nu}$ is $\nu$-admissible, the solution $c(\mathscr{L})$ exists and $c(\mathscr{L}) \in \Q^{+}.$ Finally, we define the $\nu$-th commutator multitype entry as
\begin{equation} \label{infcomm}
c_{\nu} = \inf\{ c(\mathscr{L}) \ | \ \mathscr{L} \ \textit{is $\nu$-admissible, ordered and satisfies} \ \mathscr{L}\partial r(z_0) \neq 0\}.
\end{equation}
Now, we let $z_0=0.$ From \eqref{infcomm}, we see that every finite entry of the commutator multitype is the infimum of the set of numbers $c(\mathscr{L}),$ where $\mathscr{L}$ is $\nu$-admissible, ordered, and satisfies $\mathscr{L} \partial r(0) \neq 0.$ Thus, our goal here is to show that every number $c(\mathscr{L}_1)$ associated to some crossterm is bounded from below by some number $c(\mathscr{L}_2)$ associated to a square, where $\mathscr{L}_1$ and $\mathscr{L}_2$ are lists of vector fields each satisfying the condition given in \eqref{comm} and \eqref{infcomm}. Now, if the condition $\mathscr{L} \partial r(0) \neq 0$ holds, then it must hold for at least one monomial of the boundary defining function $r$ and hence $\mathscr{L} \partial (m)(0) \neq 0,$ for some monomial $m$ in $r.$ Denote by $\mathrm{X}(m)$ the number $c(\mathscr{L})$ associated to $m$ such that  $\mathscr{L}$ satisfies the needed conditions in \eqref{infcomm} and $\mathscr{L} \partial (m)(0) \neq 0.$ We will say that the list $\mathscr{L}$ is \textit{associated to} $m,$ each time we need to make reference to the monomial $m.$ We can also see that the number $\mathrm{X}(m)$ cannot be computed if every list $\mathscr{L}$ associated to $m$ satisfies $l_{\nu}=0$ or $\sum_{i=q+2}^{\nu-1} l_i /c_i \geq 1.$ Hence, we conclude that if $\mathscr{L}$ is not $\nu$-admissible, then $\mathrm{X}(m)$ cannot be computed. We now state the following:

\begin{lemma} \label{squares}
	Let $0 \in M$ be a smooth real hypersurface in $\C^n$ with defining function $$ r = \mathrm{Re}\, z_1+ \sum_{j=1}^N |f_j(z_2, \dots, z_n)|^2,  $$ where $f_1, \ldots, f_N$ are holomorphic functions near $\mathrm{0}$ and let the Levi rank be $q$ at the origin. Let $f$ and $g$ be monomials with nonzero coefficients from the Taylor expansion of the  generator $f_j,$ $1 \leq j \leq N,$ from $M.$ Assume that for some integer $\nu,$ $q+2 \leq \nu-1 < n,$ the entries $c_1, \ldots, c_{\nu-1}$ of $\mathfrak{C}(0),$ the commutator multitype at $\mathrm{0},$ have been obtained together with the functions $r_1, r_{q+2}, \ldots, r_{\nu-1},$ and the vector fields $L_2, \ldots, L_{\nu-1}$ in $\mathfrak{B}(0),$ the boundary system at $\mathrm{0}.$
	
	\begin{itemize}
		\item[(A).] Suppose that the numbers $\mathrm{X}(|f|^2), \mathrm{X}(|g|^2),$ and $\mathrm{X}(f \bar{g})$ can be computed. 
		
		\smallskip
		
		\begin{itemize}
			\item[(i).] If $\mathrm{X}(|f|^2) = \mathrm{X}(|g|^2),$ then $\mathrm{X}(f \bar{g}) = \mathrm{X}(|f|^2) = \mathrm{X}(|g|^2).$ 
			
			\smallskip
			
			\item[(ii).] If $\mathrm{X}(|f|^2) < \mathrm{X}(|g|^2),$ then $\mathrm{X}(|f|^2)  < \mathrm{X}(f \bar{g}) <  \mathrm{X}(|g|^2).$
		\end{itemize}

		\smallskip
		
		\item[(B).] Suppose that no $\nu $-admissible ordered list $\mathscr{L}^f$ satisfying $\mathscr{L}^f \partial (|f|^2)(0) \neq 0$ exists.
		
		\smallskip 
		
		\begin{itemize}
			
			\item[(i).] No $\nu$-admissible ordered list $\mathscr{L}$ exists such that the condition $\mathscr{L} \partial (f\bar{g})(0) \neq 0$ holds if no $\nu$-admissible ordered list $\mathscr{L}^g$ satisfying $\mathscr{L}^g \partial (|g|^2)(0) \neq 0$ exists.
			
			\smallskip
			
			\item[(ii).] If there exist a $\nu $-admissible ordered list $\mathscr{L}^g$ satisfying  $\mathscr{L}^g \partial (|g|^2)(0) \neq 0$ and a $\nu $-admissible ordered list $\mathscr{L}$ satisfying $\mathscr{L} \partial (f \bar{g})(0) \neq 0,$ then 
			$$ \mathrm{X}(|g|^2) \leq \mathrm{X}(f \bar{g}).$$ 
		\end{itemize}
		
	\end{itemize}
	
\end{lemma}

\begin{proof}
	
	We will show that the next commutator multitype entry $c_{\nu}$ can always be realized  by $\mathrm{X}(|m|^2),$ the computed value of some square of a monomial $m$ from the Taylor expansion of some generator of the sum of squares domain. Suppose that $f$ and $g$ are monomials in the Taylor expansion of the generator $f_j,$ $1 \leq j \leq N,$ and let $\mathscr{L}$ be a list of vector fields associated to the crossterm $f\bar g.$ 
	
	Let $\mathscr{L} = \mathscr{L}_1^f \cup \mathscr{L}_2^{\bar{g}}, $ where $\mathscr{L}_1^f$ is the list of all holomorphic vector fields in $\mathscr{L}$ associated to $f$ and $\mathscr{L}_2^{\bar{g}}$ is the list of all antiholomorphic vector fields in $\mathscr{L}$ associated to $\bar g.$ Denote by $\mathscr{L}_1^{\bar{f}}$ and $\mathscr{L}_2^g$  the lists of all corresponding conjugate vector fields in $\mathscr{L}_1^f$ and $\mathscr{L}_2^{\bar{g}}$ respectively. Thus, for any list $\mathscr{L}$ of vector fields associated to a given crossterm $f\bar g,$ we can always generate lists $\mathscr{L}_1 = \mathscr{L}_1^f \cup \mathscr{L}_1^{\bar{f}}$ and $\mathscr{L}_2 = \mathscr{L}_2^g \cup \mathscr{L}_2^{\bar{g}},$ which are associated to the square monomials $|f|^2$ and $|g|^2$ respectively. Next, denote by $l_i$ and $\hat{l}_i$ the number of times that the vector fields $L_i$ and $\bar{L}_i$ appear in the list $\mathscr{L}.$ Similarly, denote respectively by $l_i^j$ and $\hat{l}_i^j$ the number of times that $L_i$ and $\bar{L}_i$ appear in the list $\mathscr{L}_j.$ Then  $l_i^j = \hat{l}_i^j$ for all $i = q+2, \ldots, \nu$ and $j=1,2.$  Using the expression in \eqref{comm} together with this notation, the numbers $\mathrm{X}(|f|^2),$ $\mathrm{X}(|g|^2),$ and $\mathrm{X}(f \bar{g})$ can be expressed as 
	\begin{equation} \label{3equations}
	\frac{l_{\nu }^1}{\displaystyle\frac{1}{2} - \sum_{i=q+2}^{\nu-1} \frac{l_i^1}{c_i}}, \quad \frac{l_{\nu }^2}{\displaystyle\frac{1}{2} - \sum_{i=q+2}^{\nu-1} \frac{l_i^2}{c_i}}, \ \ \text{and} \quad \frac{l_{\nu} + \hat{l}_{\nu}}{\displaystyle 1 - \sum_{i=q+2}^{\nu-1} \frac{l_i + \hat{l}_i}{c_i}}
	\end{equation}
	respectively. Also, $l_{\nu} = l_{\nu }^1$ and $\hat{l}_{\nu } = \hat{l}_{\nu }^2$ in the last expression in \eqref{3equations}.
	
	\medskip
	
	\begin{itemize}
		\item[(A).] 	If the numbers $\mathrm{X}(|f|^2), \mathrm{X}(|g|^2),$ and $\mathrm{X}(f \bar{g})$ can be computed, then we can always find $\nu$-admissible ordered lists $\mathscr{L}, \mathscr{L}_1,$ and $\mathscr{L}_2$ satisfying the conditions $\mathscr{L} \partial (f \bar{g})(0) \neq 0,$ $\mathscr{L}_1 \partial (|f|^2)(0) \neq 0,$ and $\mathscr{L}_2 \partial (|g|^2)(0) \neq 0$ respectively. From the preceding argument, it is easy to see that such lists always exist.
		
		\medskip
		
		\item[(i).]  Assume that $\mathrm{X}(|f|^2) = \mathrm{X}(|g|^2).$ Then
		
		\begin{equation} \label{Ai}
		\begin{split}
		\mathrm{X}(f \bar{g}) & = \frac{l_{\nu }^1 + \hat{l}_{\nu}^2}{\displaystyle \frac{1}{2} - \sum_{i=q+2}^{\nu-1} \frac{l_i^1}{c_i} + \frac{1}{2} - \sum_{i= q+2}^{\nu-1} \frac{\hat{l}_i^2}{c_i}} \\
		& =  \mathrm{X}(|f|^2) \left( \frac{l_{\nu }^1 + \hat{l}_{\nu }^2}{l_{\nu }^1 + l_{\nu}^2}\right)  \quad \text{from the hypothesis and \eqref{3equations}} \\
		& = \mathrm{X}(|f|^2) = \mathrm{X}(|g|^2) \quad \text {since} \quad l_{\nu }^2 = \hat{l}_{\nu }^2.
		\end{split}
		\end{equation}
		
		\item[(ii).] Assume that $\mathrm{X}(|f|^2) < \mathrm{X}(|g|^2).$ By this assumption and \eqref{3equations}, the first line of \eqref{Ai} becomes 
		
		\begin{equation} \label{Aii}
		\begin{split}
		\mathrm{X}(f \bar{g}) & >  \mathrm{X}(|f|^2) \left( \frac{l_{\nu }^1 + \hat{l}_{\nu }^2}{l_{\nu }^1 + l_{\nu }^2}\right)  \quad \text{since} \quad \frac{1}{\mathrm{X}(|f|^2)} > \frac{1}{\mathrm{X}(|g|^2)} \\
		& = \mathrm{X}(|f|^2) \quad \text {since} \quad l_{\nu }^2 = \hat{l}_{\nu }^2 \quad \text{and} \\
		\mathrm{X}(f \bar{g}) & < \mathrm{X}(|g|^2) \quad \text{ by a similar argument as in lines 1 and 2 of \eqref{Aii}.}
		\end{split}
		\end{equation}
		
		\item[(B)(i).] Suppose that every list $\mathscr{L}_1$ associated to $|f|^2$ satisfies $l_{\nu}^1 =0$ or $\sum_{i=q+2}^{\nu-1} \frac{l_i^1}{c_i} \geq \frac{1}{2}$ and similarly, every list $\mathscr{L}_2$ associated to $|g|^2$ satisfies $l_{\nu}^2 =0$ or $\sum_{i=q+2}^{\nu-1} \frac{l_i^2}{c_i} \geq \frac{1}{2}.$ We break the remaining argument into 4 cases. Suppose that $\mathscr{L}_1$ and $\mathscr{L}_2$ are such that:
		
		\smallskip
		
		\begin{itemize}
			\item[CASE \large 1.] $l_{\nu}^1 = l_{\nu}^2 = 0.$ Clearly, the list $\mathscr{L}$ cannot be $\nu$-admissible since $l_{\nu }^1 + \hat{l}_{\nu}^2 =0.$
			
			\smallskip
			
			\item[CASE \large 2.] $\sum_{i=q+2}^{\nu-1} \frac{l_i^1}{c_i} \geq \frac{1}{2}$ and $\sum_{i=q+2}^{\nu-1} \frac{l_i^2}{c_i} \geq \frac{1}{2}.$ Then $\sum_{i=q+2}^{\nu-1} \frac{l_i^1 + \hat{l}_i^2}{c_i} \geq 1$ and so $\mathscr{L}$ is not $\nu$-admissible.
			
			\smallskip
			
			\item[CASE \large 3.] $l_{\nu }^1=0$ and $\sum_{i=q+2}^{\nu-1} \frac{l_i^2}{c_i} \geq \frac{1}{2}.$ Now, suppose that $\mathscr{L}$ is $\nu$-admissible and satisfies $\mathscr{L} \partial (f \bar{g})(0) \neq 0.$ Clearly, the list $\mathscr{L}_1$ satisfies $l^1_{\nu} = \hat{l}^1_{\nu} =0.$ By the assumption on $\mathscr{L}$ and the selection criteria of its vector fields, $\mathscr{L}_1$ satisfies $\sum_{i=q+2}^{\nu-1} \frac{l_i^1}{c_i} = \frac{1}{2}.$ Thus, $\sum_{i=q+2}^{\nu-1} \frac{l_i^1 + \hat{l}_i^2}{c_i} \geq 1$ contradicting the fact that 
			$\mathscr{L}$ is $\nu$-admissible.	
			
			\smallskip
			
			\item[CASE \large 4.] $l_{\nu}^2=0$ and $\sum_{i=q+2}^{\nu-1} \frac{l_i^1}{c_i} \geq \frac{1}{2}.$ The argument here follows similarly as in case 3 above.
		\end{itemize}
		
		\smallskip
		
		\item[(ii).] Suppose that $\mathscr{L}$ is $\nu$-admissible, ordered, and satisfies $\mathscr{L} \partial (f \bar{g})(0) \neq 0$ and that there exists a $\nu$-admissible, ordered list $\mathscr{L}_2$ satisfying $\mathscr{L}_2 \partial (|g|^2)(0) \neq 0.$ If no $\nu$-admissible, ordered list $\mathscr{L}_1$ exists such that the condition $\mathscr{L}_1 \partial (|f|^2)(0) \neq 0$ holds, then we shall break the argument into 2 cases.
		
		\begin{itemize}
			\item[CASE \large 1.] Suppose that every list $\mathscr{L}_1$ satisfies $\sum_{i=q+2}^{\nu-1} \frac{l_i^1}{c_i} \geq \frac{1}{2}.$ Then line 1 of \eqref{Ai} becomes 
			
			\begin{equation} \label{BiC1}
			\begin{split}
			\mathrm{X}(f \bar{g}) & \geq \frac{l_{\nu}^1 + \hat{l}_{\nu}^2}{\displaystyle \frac{1}{2} - \sum_{i= q+2}^{\nu-1} \frac{\hat{l}_i^2}{c_i}} \quad \text{since} \quad \frac{1}{2} - \sum_{i= q+2}^{\nu-1} \frac{\hat{l}_i^1}{c_i} \leq 0, \\
			& \geq \mathrm{X}(|g|^2) \quad \text{from \eqref{3equations} and the inequality} \quad \frac{l_{\nu }^1 + \hat{l}_{\nu}^2}{l_{\nu}^2} \geq 1.
			\end{split}
			\end{equation}
			
			\item[CASE \large 2.] Suppose that every list $\mathscr{L}_1$ satisfies $l_{\nu}^1 = \hat{l}_{\nu}^1 =0.$ By the assumptions on $\mathscr{L}$ and $\mathscr{L}_2$ as well as the selection criteria of vector fields in $\mathscr{L},$ the list $\mathscr{L}_1$ satisfies $\sum_{i=q+2}^{\nu-1} \frac{l_i^1}{c_i} = \frac{1}{2}.$ Then from \eqref{3equations} and the equation $\frac{1}{2} - \sum_{i= q+2}^{\nu-1} \frac{\hat{l}_i^1}{c_i} = 0,$ line 1 of \eqref{Ai} satisfies $\mathrm{X}(f \bar{g}) = \mathrm{X}(|g|^2).$
		\end{itemize}
	\end{itemize}
	
	This completes the proof of the lemma.
\end{proof}

\begin{remark} \label{crsrmk}
	It is easy to see that $\mathrm{X}(f \bar{g}) = \mathrm{X}(\bar{f} g).$ Also, as a consequence of Lemma \ref{squares}, the omission of all the crossterms from the expansion of the moduli squares of the generators does not alter the value of the commutator multitype.

\end{remark}

Lemma \ref{squares} enables us to compute the commutator multitype of a sum of squares domain from its associated ideal of holomorphic functions. As a result of this transition, we will give an ideal restatement of Catlin's machinery for the commutator multitype computation in section 6.

\section{Derivatives of the Ideal of Holomorphic Functions}

Let $\left< f \right> = \left< f_1, \dots, f_N \right>$	be the ideal of holomorphic functions associated to the sum of squares domain $D$ and denote by $\text{J}(f) \text{J}^{*}(f)$ the Levi matrix of $M,$ where $\text{J}(f)$ and $\text{J}^{*}(f)$ are the complex Jacobian matrix and its conjugate transpose respectively. We know that the commutator multitype at 0 is obtained via differentiation of the Levi form of M at 0 along certain selected directions. To each vector field $L_k= \sum_{j=1}^n v_{kj} \frac{\partial}{\partial z_j} \in T^{1,0}M,$ we associate a row vector $v_k = (v_{k1}, \ldots, v_{kn}).$ We set $r=r_1,$ where $r$ is the boundary-defining function associated to $D.$ Since $r_1$ defines the boundary of some manifold we associate the vector $v_1 = (1, 0, \dots, 0)$ to the vector field $L_1$ given that $L_1r_1=1$ is always satisfied. Denote by $\left[ \partial r_1\right]$ the $1 \times n $ matrix 
\begin{equation} \label{normaldirection}
\left[ \partial r_1\right] :=
\begin{pmatrix}
\partial_{z_1}r_1 & \partial_{z_2}r_1 & \cdots & \partial_{z_n}r_1
\end{pmatrix}, 
\end{equation}
where $\partial_{z_i}:= \frac{\partial}{\partial z_i}$ and define the kernel of $\left[ \partial r_1\right]$ as
\[
\ker \left[ \partial r_1 \right] = \{v \ | \  \left[ \partial r_1 \right] v^{T} = 0\},
\]
where $v^{T}$ is the transpose of $v,$ the associated row vector to any vector field $L \in T^{1,0}M.$ So the condition $L_i(r_1) \equiv 0$ for all $i, \ 2 \leq i \leq n,$ given in \cite{catlinbdry} is equivalent to the condition that $v_i \in \ker \left[ \partial r_1 \right],$ where $v_i$ is the row vector associated to $L_i \in T^{1,0}M.$ Now, for every fixed $k,$ let $L_k  = \sum_{j=1}^n v_{kj}\frac{\partial}{\partial z_j}$ with $v_{kj}$ a holomorphic function for every $j$ be a vector field in $T^{1,0}M.$ We shall only consider such holomorphic vector fields to define the derivatives of the ideal of holomorphic functions. For the rest of the paper, unless otherwise specified, every holomorphic vector field $L_k$ can be viewed as a $1 \times (n-1)$ matrix by associating with the row vector $v_k = \left( v_{k2}, \ldots, v_{kn} \right)$ since the holomorphic functions $f_1, \dots, f_N$ do not depend on the variable $z_1$ and $T^{1,0}M = \mathrm{span} \left\{\frac{\partial}{\partial z_2}, \dots, \frac{\partial}{\partial z_n}  \right\}.$

Using the complex Jacobian matrix associated to the ideal of holomorphic functions $\left\langle  f\right\rangle,$ we will define the derivative of $\left\langle f\right\rangle $ along directions belonging to some ordered list of holomorphic vector fields in $T^{1,0}M.$ Recall that we are interested in the derivatives of the Levi form along certain vector fields. Now, let $L \in T^{1,0} M.$ In Catlin's notation \cite{catlinbdry}, the Levi form of M at a point 0 along the direction $L$ is given by $\partial \bar{\partial} r([L, \bar{L}])(0).$ For any given sum of squares domain, 
\begin{equation} \label{eqlevi}
\partial \bar{\partial} r([L, \bar{L}])(z) \ \equiv \  v \text{J}(f) \text{J}^{*}(f) v^{*}(z),
\end{equation}
where $v$ is the row vector associated to $L.$ A simple verification shows that \eqref{eqlevi} is likewise a sum of squares. Specifically, $v \text{J}(f) \text{J}^{*}(f) v^{*}(z)$ is a sum of squares with generators $Lf_j = \sum_{l=2}^n v_l (\partial f_j / \partial z_l),$ $j = 1, \ldots, N$ holomorphic near the origin. So the Levi form of a sum of squares domain in the direction $L$ has an ideal representation, which we denote by $\left\langle Lf \right\rangle:= \left\langle Lf_1, \ldots, Lf_N\right\rangle.$ Now, $L_t \partial \bar{\partial} r([L, \bar{L}])(z),$ the derivative of $\partial \bar{\partial} r([L, \bar{L}])(z)$ along the direction $L_t \in T^{1,0}M$ is equivalent to the form $v_t \text{J}(Lf) \text{J}^{*}(f) v^{*}(z),$ where $\text{J}(Lf)$ is the complex Jacobian matrix associated to the ideal $\left\langle Lf \right\rangle.$ Similarly, the derivative of $\partial \bar{\partial} r([L, \bar{L}])(z)$ along the direction $\bar{L}_t$ becomes $v \text{J}(f) \text{J}^{*}(Lf) v_t^{*}(z)$ and the derivative of the sum $\sum_{j=1}^N |f_j|^2$ in the direction $\bar{L}_t$ is given by 
\begin{equation} \label{1st step}
\begin{pmatrix} f_1 & \cdots & f_N  \end{pmatrix} \text{J}^{*}(f) v_t^{*}(z),
\end{equation}
where $\begin{pmatrix} f_1 & \cdots & f_N  \end{pmatrix}$ is the matrix associated to the ideal of holomorphic functions $\left\langle f\right\rangle .$ The form given in \eqref{1st step} will be useful in the model domain case of the ideal reformulation of the commutator multitype  in section 6. From Lemma \ref{squares}, we know that every finite entry of the commutator multitype at 0 can always be realized by the modulus square of a monomial in some generator of the ideal of holomorphic functions. Thus each entry of the commutator multitype can be realized by some list $\mathscr{L} = \{L^1, \ldots, L^l\}$ of vector fields where $l$ is even. Most importantly, there exists a list $\mathscr{L} = \{L^1, \ldots, L^l\}$  such that for every $t,$  the vector fields $L_t$ and $\bar{L}_t$ in $\mathscr{L}$ must appear an equal number of times. Using this fact together with the fact that $v \text{J}(f) \text{J}^{*}(f) v^{*}(z)$ is a sum of squares, we shall define the derivative of an ideal of holomorphic functions. We first note that $\left\langle Lf \right\rangle:= \left\langle Lf_1, \ldots, Lf_N\right\rangle,$ the ideal representation of the Levi form of M is precisely the first order derivative of the ideal $\left\langle f\right\rangle $ along the direction $L.$ We associate the $1 \times N$ matrix $v \text{J}(f) = \begin{pmatrix}
Lf_1 & \cdots & Lf_N
\end{pmatrix}$ to the ideal $\left\langle Lf\right\rangle .$ Before we define the higher order derivatives of the ideal $\left\langle f\right\rangle $ we consider the definition below:

\begin{definition} \label{equalorder}
	Let $\mathscr{L} = \{ L^1, \ldots, L^l \}$ be a list of vector fields in $T^{1,0}M$ such that $L^j$ is one of  $\{L_{c}, \bar{L}_{c} \ldots, L_{c+t}, \bar{L}_{c+t} \}$ for some $c, t \in \Z_{>0}$ and $L_i,$ $i= c, \ldots, c+t,$ is a holomorphic vector field in $T^{1,0}M.$ Suppose that $L_i$ and $\bar{L}_i$ appear an equal number of times in $\mathscr{L}$ and so each $l_i$ is even for $i= c, \ldots, c+t.$ Then $\mathscr{L}$ is said to be \textit{equally-ordered} if 
	
	\begin{equation} \label{veclist3}
	L^{2j-1} = L_{c+t} \ \ \text{and} \ \ L^{2j} = \bar{L}_{c+t} \ \ \text{when} \quad  1 \leq j \leq  \frac{1}{2}(l_{c+t}), \quad \text{and} 
	\end{equation}
	
	\begin{equation} \label{veclist4}
	L^{2j-1} = L_i \ \ \text{and} \ \ L^{2j} = \bar{L}_i \ \ \text{when} \quad \frac{1}{2} \left(2 + \sum_{k=i+1}^{c+t} l_k \right) \leq j \leq  \frac{1}{2} \left(\sum_{k=i}^{c+t} l_k \right).
	\end{equation}
\end{definition}

\begin{remark} \label{equalorderrk}
	Let $\mathscr{L} = \{ L^1, \ldots, L^l \}$ be an equally-ordered list of vector fields in $T^{1,0}M$ as in Definition \ref{equalorder}. Then we will denote by $\mathscr{L}_{\text{h}}= \{ L^1, \ldots, L^{s} \}$ for $s= \frac{l}{2},$ the ordered list (in the sense of Catlin) of holomorphic vector fields in the list $\mathscr{L}$ in Definition \ref{equalorder}.	
\end{remark}

Now, it is easy to see by a similar argument as in the case of the Levi form, that by differentiating $v \text{J}(f) \text{J}^{*}(f) v^{*}(z)$ along the vector fields in an equally-ordered list $\mathscr{L} =\{ L^1, \ldots, L^l \}$ an even number of times, we still obtain a function that is a sum of squares. More specifically, let $j , d \in \Z_{> 0}$ be such that $1 \leq j \leq d$ and $d = \frac{1}{2}\sum_{k=c}^{c+t} l_k.$ Then after differentiating $2j$ times $v \text{J}(f) \text{J}^{*}(f) v^{*}(z)$ along vector fields in $\mathscr{L}$ we obtain the function 
\begin{equation} \label{derofL}
v^m \text{J}(F^j) \text{J}^{*}(F^j) (v^m)^{*}(z),
\end{equation}
where $m = 2(d-j+1)-1,$ $\displaystyle F^j = \prod_{i=2}^{j} L^{2(d-j+i)-1}(Lf),$ and $v^m$ is the row vector associated to the vector field $L^m \in \mathscr{L}.$ It is easy to see that \eqref{derofL} is likewise a sum of squares. Let $\displaystyle \tilde{F}^j = \prod_{i=1}^{j} L^{2(d-j+i)-1}(Lf).$ Then to the ideal of holomorphic functions $\big\langle \tilde{F}^j \big\rangle,$ we associate the $1 \times N$ matrix  
\begin{equation} \label{derofL2}
v^m \text{J}\left(F^j\right) = \begin{bmatrix}
\tilde{F}^j_1 & \tilde{F}^j_2 & \cdots & \tilde{F}^j_N
\end{bmatrix},
\end{equation}
where $\tilde{F}^j_k$ is the derivative of the function $\tilde{F}^j$ with respect to the variable $z_k.$ This allows us to define the higher-order derivatives of the ideal $\left\langle f\right\rangle$ along vector fields in the list of holomorphic vector fields $\mathscr{L}_{\text{h}}.$ We will use exponents with parentheses to indicate the order of the derivative of any vector field in order to distinguish them from the usual indexing of vector fields.

\begin{definition}
	Let $L^j = \sum_{i=2}^n v_{ji} \frac{\partial}{\partial z_i}$ be a holomorphic vector field in $T^{1,0} M$ and let $\left< f\right> = \left< f_1, \dots, f_N \right>$ be an ideal of holomorphic functions. 
	
	\begin{itemize}
		\item[(i).] Let $\mathscr{L}_{\text{h}} = \{L^{1}, \ldots, L^{k}\}$ be an ordered list of holomorphic vector fields in $T^{1,0}M$. Then the kth order derivative of the ideal $\left < f \right>$ in the k directions in $\mathscr{L}_{\text{h}}$ is given by $$\left< (L^{1} \cdots L^{k}) f\right> := \left< (L^{1} \cdots L^{k}) f_1, \dots , (L^{1} \cdots L^{k}) f_N \right>.$$
		
		\item[(ii).] If the vector fields $L^j=L$ in (i) for some $L \in T^{1,0} M$, then we define the kth order derivative of the ideal $\left< f\right>$ in the direction of $L$ by
		$$  \left< L^{(k)} f\right> : = \left< L^{(k)} f_1, \dots, L^{(k)} f_N\right>.$$

	\end{itemize}
\end{definition}
\noindent By a slight abuse of notation, we will denote $\left< (L^{1} \cdots L^{k}) f\right>$ by $\left\langle  \mathscr{L}_{\text{h}} f\right\rangle$ and its corresponding complex Jacobian matrix by $\text{J}(\mathscr{L}_{\text{h}} f).$ In this notation, the order of the derivative of $\left\langle f\right\rangle $ is given by $\# \mathscr{L}_{\mathrm{h}},$ the cardinality of $\mathscr{L}_{\mathrm{h}}.$ 

\begin{remark} \label{jacob} 
	Let $\mathscr{L}^{'}_{\text{h}}:= \{L^{2}, \ldots, L^{k}\},$ $\mathscr{L}^{''}_{\mathrm{h}} := \{L^{3}, \ldots, L^{k}\}$ and note the following:
	
	\begin{itemize}
		\item[(i).]  $  \left\langle  \mathscr{L}_{\text{h}} f\right\rangle = \left< L^1 \cdots L^k f\right>$ is the ideal representation of the function $$ v^1 \mathrm{J}(\mathscr{L}_{\mathrm{h}}^{'}f) \mathrm{J}^{*}(\mathscr{L}_{\mathrm{h}}^{'}f) (v^1)^{*}(z).$$
		\item[(ii).] Suppose that $\mathscr{L}^{'}_{\mathrm{h}} = \emptyset.$ By convention, we set $\mathrm{J}(\mathscr{L}^{'}_{\mathrm{h}} f) = \text{J}(L^{(0)}f) = \mathrm{J}(f)$ and $ \ v \mathrm{J}(\mathscr{L}^{''}_{\mathrm{h}} f) = \begin{pmatrix} f_1& \cdots & f_N \end{pmatrix}.$
		\item[(iii).] We simply let $L^{(1)} f = L f$ when no confusion arises.
	\end{itemize}
\end{remark}

\medspace

\section{Normalizing the Boundary Systems}

We give the proofs of the main theorems in this section. The proof of Theorem \ref{equalm} relies on the normalization results in Theorem \ref{norm}. The proof of Theorem \ref{norm} uses results from Corollary 1.2 in \cite{bdn} and Lemma \ref{rshol} below. For easy reference, we state without proof Corollary 1.2 as given in \cite{bdn} as Corollary \ref{corinbnd} below:

\begin{corollary}[\cite{bdn}, pg. 3] \label{corinbnd}
	Let $M$ be a pseudoconvex smooth real hypersurface in $\C^n$ with $0 \in M$ and of the Catlin multitype $$ \Lambda = (1, \lambda_2, \dots, \lambda_n), \quad \lambda_n < \infty$$ at $0.$
	
	Then there exists a holomorphic change of coordinates at $\mathrm{0}$ preserving the multitype so that the defining function for $M$ in the new coordinates is given by \[
	r = 2\mathrm{Re} z_1 + p(z_2, \dots, z_n, \bar z_1, \dots, \bar z_n) +o_{ \Lambda^{-1}}(1),
	\]
	where $p$ is a polynomial consisting of all terms of weight $1$ with respect to the Catlin multitype at $0$ and also contains the balanced sum of squares 
	\[
	|z_2|^{2k_{22}}+ |z_2|^{2k_{32}} |z_3|^{2k_{33}} + \cdots + |z_2|^{2k_{n2}} \cdots |z_n|^{2k_{nn}},
	\]
	with $k_{jj} >0$ for all $j$ and the degree of $p$ in each $(z_j, \bar z_j)$ is not greater than $2k_{jj}.$
\end{corollary}

Before we give the proof of the complete normalization of the boundary system for a sum of squares domain, we give an answer to question (i) in the introduction -namely,  can torsion occur in the boundary of any sum of squares domain in $\C^n?$ The ensuing lemma gives a negative answer to this question by establishing that no torsion can occur in the boundary. Recall that by Theorem 1.1 in \cite{Aidoo2022},  the model polynomial of a sum of squares domain is likewise a sum of squares. Hence, we state the following:

\begin{lemma} \label{rshol}
	Let $M_0 \subset \C^n \ n \geq 2$ be a model hypersurface of a model sum of squares domain $D_0$  defined by $\{r_0 < 0\},$ where
	\[
	r_0 = 2\mathrm{Re} z_1 + \sum_{l=1}^N|h_l(z_2, \dots, z_n)|^2,
	\]
	and $h_l$ is a polynomial consisting of all terms from the Taylor expansion of $f_l$ of weight $1/2$ with respect to the Catlin multitype at $0,$
	\[
	\Lambda = \left(1, \lambda_2, \lambda_3, \dots, \lambda_n\right), \quad \lambda_n < + \infty.
	\]	
	For any monomial $f$ with nonzero coefficients from the Taylor series expansion of the generator $h_l$ of $r_0,$ for $ 1 \leq l \leq N,$ the function
	\begin{equation}\label{holeq}
	F(z, \bar z) = D^{\alpha_0^f} \bar{D}^{\alpha^f} r_0,
	\end{equation}
	is holomorphic. Here, $\alpha^f=(\alpha^f_{i_1}, \dots, \alpha^f_{i_k})$ for some $k \leq n,$  
	is the multiindex corresponding to the monomial $f$ 
	and $\alpha_0^f= (\alpha^f_{i_1}-1, \alpha^f_{i_2} \dots, \alpha^f_{i_k}).$ 
\end{lemma}

\begin{proof}
	The normalization argument in Lemma 6.2 of \cite{bdn} amounts to showing that taking one fewer holomorphic differentiation on the model polynomial than necessary to obtain a non-zero quantity in the computation of the commutator multitype yields a holomorphic function. Showing that $F(z, \bar{z})$ is holomorphic accomplishes the same for a sum of squares domain in this lemma.
	
	Let $f$ and $g$ be distinct monomials with nonzero coefficients from the Taylor series expansion of some generator $h_l$ of $r_0$ for $l$ fixed with $1 \leq l \leq N.$ If there are variables in $f$ not present in $g,$ then immediately $\bar{D}^{\alpha^f} \bar{g} =0,$ so we only need to consider monomials $g$ such that 
	\begin{equation}\label{fxns}
	f=C_f z_{i_1}^{\alpha_{i_1}^f} \cdots z_{i_k}^{\alpha_{i_k}^f} \quad \text{and} \quad g=C_g z_{i_1}^{\alpha_{i_1}^g} \cdots z_{i_k}^{\alpha_{i_k}^g} z_{i_{k+1}}^{\alpha_{i_{k+1}}^g} \cdots z_{i_c}^{\alpha_{i_c}^g}
	\end{equation}
	for some $c, k$ satisfying $k\leq c \leq n$ where $\alpha^g=(\alpha^g_{i_1}, \dots, \alpha^g_{i_c})$ is the  multiindices corresponding to $g$ and $C_f, \, C_g \in \C.$
	Also, let $|\alpha^f|=\alpha_{i_1}^f+ \cdots + \alpha_{i_k}^f,$ $|\alpha^g|=\alpha_{i_1}^g+ \cdots + \alpha_{i_k}^g,$ and $|\alpha^f_0|=\alpha_{i_1}^f+\alpha_{i_2}^f + \cdots + \alpha_{i_k}^f-1.$ From the expansion of $|h_l|^2$ we obtain the terms $|f|^2,$ $|g|^2,$ and $\mathrm{2Re}(f\bar g)$ in $r_0$ satisfying 
	\begin{equation} \label{1steq}
	\sum_{j=1}^k \frac{\alpha_{i_j}^f}{\lambda_{i_j}} = \frac{1}{2}, \quad  \sum_{j=1}^c \frac{\alpha_{i_j}^g}{\lambda_{i_j}} = \frac{1}{2} \quad \text{and} \quad \sum_{j=1}^k \frac{\alpha_{i_j}^f}{\lambda_{i_j}} + \sum_{j=1}^c \frac{\alpha_{i_j}^g}{\lambda_{i_j}} = 1
	\end{equation}
	respectively. To show that the function $F(z, \bar z)$ in \eqref{holeq} is holomorphic, it suffices to show that 
	$\bar{D}^{\alpha^f} \bar g \equiv 0$ and $\bar{D}^{\alpha^f} \bar u \equiv 0,$ where $u$ is any distinct monomial from the Taylor series expansion of a generator different from $h_l$ in $r_0.$ Now, if for some $j, \ 1 \leq j \leq k,$ the inequality $\alpha_{i_{j}}^f > \alpha_{i_{j}}^g$ holds, then it is easy to see that $\bar{D}^{\alpha^f} \bar g \equiv 0.$ Thus, we only need to examine monomials $f$ and $g$ as given in \eqref{fxns} such that for every $j,$ $1 \leq j \leq k,$ $\alpha_{i_{j}}^f \leq \alpha_{i_{j}}^g.$ 
	We shall now break the argument into two cases:
	
	\begin{itemize}
		\item[CASE 1.] Suppose that $k = c$ in the definitions given in \eqref{fxns}. Then due to the weight restrictions in \eqref{1steq},  $\alpha_{i_j}^f=\alpha_{i_j}^g$ for all $j=1,2, \dots, k$ or $\alpha_{i_j}^f \neq \alpha_{i_j}^g$ for some $j, \ 1 \leq j \leq k.$ Note that we cannot have  $\alpha_{i_j}^f < \alpha_{i_j}^g$ for all $j, \ 1 \leq j \leq k,$ since this will violate the fact that both $|f|^2$ and $|g|^2$ are in the model. Now, we consider the following:
		
		\smallskip
		
		\begin{itemize}
			\item[i.] First, suppose that for all $j=1,2, \dots, k,$ $\alpha_{i_j}^f=\alpha_{i_j}^g.$ Then $f = g$ contradicting the assumption that $f$ and $g$ are distinct.
			
			\item[ii.]  Next, if $\alpha_{i_j}^f < \alpha_{i_j}^g$ for some $j,$ then $\alpha_{i_{j^{'}}}^f > \alpha_{i_{j^{'}}}^g$ for some other $j^{'}$ because otherwise $f$ and $g$ cannot both have weight $\frac{1}{2}$ as seen in \eqref{1steq}. Therefore,  $\bar{D}^{\alpha^f} \bar g \equiv 0.$
			
		\end{itemize}
		
		\smallskip
		
		\item[CASE 2.]	Suppose that $k < c$ in the definitions given in \eqref{fxns}. Then $\alpha^g_{i_j} \neq 0$ for $j= k+1, \dots, c$ so $\alpha^g_{i_j} < \alpha^f_{i_j}$ for at least one $j,$ $1 \leq j \leq k$ since $f$ and $g$ have equal weight $\frac{1}{2}$ as seen in \eqref{1steq}. Hence, $\bar{D}^{\alpha^f} \bar g \equiv 0.$ 
	\end{itemize}
	
	Now, let $u$ be any distinct monomial from the Taylor series expansion of a generator different from $h_l$ in $r_0$ and suppose that $u$ and $f$ are distinct. Then since $f$ and $g$ are also distinct monomials, it follows that by arguing in the same manner with $u$ replacing $g,$ we can conclude that $\bar{D}^{\alpha^f} \bar u \equiv 0.$ Consequently, for any given pair of distinct monomials $f$ and $g$ from $r_0,$  the corresponding square $|g|^2$ and crossterm $f\bar g$ vanish identically after applying the operator $D^{\alpha_0^f} \bar{D}^{\alpha^f}$ to $r_0.$ Therefore, applying the operator $D^{\alpha_0^f} \bar{D}^{\alpha^f}$ to $|f|^2$ and $\bar f g$ gives $Az_{i_1}$ and $C\tilde{g}$ respectively, where $A, C \in \C$ and $\tilde{g}=D^{\alpha_0^f}g$ is holomorphic. Finally, 
	\begin{equation}\label{final}
	F(z, \bar z) = D^{\alpha_0^f} \bar{D}^{\alpha^f} r_0 = K z_{i_1} + \Psi_{i_1}, 
	\end{equation}
	where $\Psi_{i_1}= \sum_{l=1}^m K_l D^{\alpha_0^f} \psi_l,$ $m \in \Z_{> 0},$ $K, K_l \in \C, \ l=1, \dots, m,$ and $\psi_l $ is a holomorphic monomial from some generator of $\sum_{l=1}^N |h_l|^2$ such that $\psi_l$ and $f$ are distinct. Due to the weight restrictions given in \eqref{1steq}, if $\psi_l \not\equiv 0,$ then $\psi_l$ has weight $1/2.$ $F(z, \bar z)$ is holomorphic as required.
	
\end{proof}

\begin{remark} \label{holrmk}
	The term $\Psi_{i_1}$ in \eqref{final} is independent of the variable $z_{i_1}$ due to weight restrictions.
\end{remark}

\medspace

As hinted in the introduction of this section, the proof of our first main theorem heavily relies on Corollary \ref{corinbnd} and Lemma \ref{rshol} above. Armed with these tools we now give the proof of Theorem \ref{norm}.

\bigskip

\noindent {\bf Proof of Theorem \ref{norm}.} Let $M$ be a sum of squares hypersurface in $\C^n$ with defining function
\[
r = 2\mathrm{Re} z_1 + \sum_{l=1}^N|f_l(z_2, \dots,z_n)|^2,
\]
where $f_l(z_2, \dots,z_n)$ is a holomorphic function near the origin and assume that the rank of the Levi form of $M$ at 0 is $q.$ From Corollary \ref{corinbnd} there exists a holomorphic change of coordinates at 0 preserving the multitype so that the defining function $r$ in the new coordinates is given by
\[
r = 2\mathrm{Re} z_1 + p(z_2, \dots, z_n, \bar z_1, \dots, \bar z_n) +o_{ \Lambda^{-1}}(1),
\]
where $p$ is a polynomial consisting of all terms of weight $1$ with respect to the Catlin multitype at $0$ and also contains the balanced sum of squares 
\begin{equation} \label{bsos}
|z_2|^{2k_{22}}+ |z_2|^{2k_{32}} |z_3|^{2k_{33}} + \cdots + |z_2|^{2k_{n2}} \cdots |z_n|^{2k_{nn}},
\end{equation}
with $k_{jj} >0$ for all $j$ and the degree of $p$ in each $(z_j, \bar z_j)$ is not greater than $2k_{jj}.$ Note that since $M$ is a sum of squares domain, all holomorphic changes of coordinates will take squares to squares. The proof here reduces to proving that at the level of the model hypersurface $r_0=0,$  where
\[
r_0 = 2\mathrm{Re} z_1 + p(z_2, \dots, z_n, \bar z_1, \dots, \bar z_n),
\]
for $q+2 \leq j \leq n,$ the real-valued function $r_j$ in the boundary system $\mathfrak{B}_n(0)$ can be transformed into the form $\mathrm{Re}\, z_j$ via a holomorphic polynomial change of coordinates. Via a Chern-Moser type
argument \cite{CM}, we assume that the Levi rank at the origin is $q = 0.$ Also, from \cite{Aidoo2022} we know that $p$ is a sum of squares and so let $p=\sum_{l=1}^N|h_l(z_2, \dots, z_n)|^2,$ where $h_l$ a holomorphic function of weighted degree 1/2. We proceed further by rewriting the balanced sum of squares in \eqref{bsos} as 
\begin{equation} \label{bsos2}
|z_2^{k_{22}}|^2+ |z_2^{k_{32}} z_3^{k_{33}}|^2 + \cdots + |z_2^{k_{n2}} z_3^{k_{n3}}\cdots z_n^{k_{nn}}|^2
\end{equation}
and applying Lemma \ref{rshol} to the defining function $r_0$ as follows: Denote by $f(j)$ the monomial given by 
\[
f(j) = z_2^{k_{j2}} z_3^{k_{j3}} \cdots z_j^{k_{jj}}, \quad \text{for all} \quad 2 \leq j \leq n,
\] 
and let $\alpha_0^{f(j)} = (k_{j2}, k_{j3}, \dots, k_{jj}-1)$ and $\alpha^{f(j)} = (k_{j2}, k_{j3}, \dots, k_{jj}).$ Note here that the definition of the multiindices $\alpha_0^{f(j)}$ still holds since $k_{jj} >0.$ From Lemma \ref{rshol},
\begin{equation} \label{finalnorm}
D^{\alpha_0^{f(j)}} \bar{D}^{\alpha^{f(j)}} r_0 = C_j z_j + \Psi_j, \quad \text{for all} \quad 2 \leq j \leq n,
\end{equation}
where $\Psi_j$ is a holomorphic function. From Remark \ref{holrmk}, $\Psi_j$ is independent of $z_j$ for each $j,$ $2 \leq j \leq n,$ and so we can make a suitable holomorphic change of coordinates. Hence, for each fixed $j,$ $2 \leq j \leq n,$ we obtain an equally ordered list $\mathscr{L}$ given by
\[
L_j: = \partial_{z_j} + p_{z_j} \partial_{z_1}, \quad \mathscr{L} = \{\underbrace{L_j, \bar L_j, \dots, L_j, \bar L_j}_{2k_{jj}}, \dots, \underbrace{L_2, \bar L_2, \dots, L_2, \bar L_2}_{2k_{22}} \},
\]
and 
\[
\mathscr{L}{'} = \{\underbrace{\bar L_j, \dots, L_j, \bar L_j}_{2k_{jj}-1}, \dots, \underbrace{L_2, \bar L_2, \dots, L_2, \bar L_2}_{2k_{22}} \}.
\]
Hence we compute $r_j$ as 
\begin{equation} \label{finalnorm2}
r_j = \mathscr{L}^{'}\partial r \sim \mathrm{Re} D^{\alpha_0^{f(j)}} \bar{D}^{\alpha^{f(j)}} r_0 \sim \mathrm{Re}(C_j z_j + \Psi_j).
\end{equation}
Finally, we make a holomorphic change of coordinates that preserves $\Lambda:$ $z_l \to \tilde z_l$ for $1 \leq l \leq n,$ where $\tilde z_l = C_l z_l + \Psi_l$ for $l=j,$ and $\tilde z_l = z_l$ for $l \neq j.$ Thus, in the new coordinates (after normalizing the constants to 1), the real-valued function $r_j$ in \eqref{finalnorm2}  becomes
\begin{equation} \label{finalnorm3}
r_j = \mathrm{Re}\, z_j,
\end{equation}
for all $2 \leq j \leq n.$ 
\begin{flushright} 
	\qed
\end{flushright}	

\bigskip

We shall prove Theorem \ref{equalm} by establishing the two inequalities $\mathfrak{M}(0) \leq \mathfrak{C}(0)$ and $\mathfrak{M}(0) \geq \mathfrak{C}(0).$ The proof of the first inequality is easy and is also given in \cite{catlinbdry} Theorem 3.7, for a general pseudoconvex domain in $\C^n$. We are therefore interested in showing that $\mathfrak{M}(0) \geq \mathfrak{C}(0)$ for any sum of squares domain. The proof of the second inequality for a general pseudoconvex domain is also given by Catlin in \cite{catlinbdry} via a simultaneous inductive argument, but the details of the proof are very difficult to read. Hence, for any given domain of sum of squares of holomorphic functions, we shall present a very simple proof of the second inequality using results from Theorem \ref{norm}.

\bigskip

\noindent {\bf Proof of Theorem \ref{equalm}.} Let $0 \in M$ be a smooth real hypersurface in $\C^n$ with defining function $$ r = \mathrm{Re}\, z_1+ \sum_{j=1}^N |f_j(z_2, \dots, z_n)|^2,  $$ where $f_1, \ldots, f_N$ are holomorphic functions near $\mathrm{0}$ and suppose that $\mathfrak{B}_n(0)$ is the boundary system defined near the origin. Let $\mathfrak{C}(0)$ be the commutator multitype at 0 associated to $\mathfrak{B}_n(0)$ and  suppose that both $\mathfrak{C}(0)$ and $\mathfrak{M}(0),$ the multitype at 0, are finite. 

We will show that $\mathfrak{M}(0) \geq \mathfrak{C}(0).$ By definition, the multitype is the infimum of all distinguished weights and so $\mathfrak{M} \geq \Lambda,$ where $\Lambda \in \tilde{\Gamma},$ the set of all distinguished weights. Therefore, it suffices to show that $\mathfrak{C}(0)$ is a distinguished weight. From Corollary \ref{corinbnd} applied in the proof of Theorem \ref{norm}, there exists a holomorphic change of coordinates that transforms the defining function $r$ into the form given in \eqref{distgeq} in Remark \ref{distgrk}. The complete normalization result from Theorem \ref{norm} ensures that there exists a holomorphic change of coordinates, which maps $(z_1, \dots, z_n)$ to the origin and also transforms $r$ into the form in \eqref{distgeq}. Hence, the weight $\mathfrak{C}(0)$ at 0, generated from $\mathfrak{B}_n(0)$ is a distinguished weight. This completes the proof of Theorem \ref{equalm}.
\begin{flushright} 
	\qed
\end{flushright}
\medspace

\section{Ideal Restatement of the Commutator multitype machinery}

We reformulate Catlin's technique for the commutator multitype computation in terms of the ideal of holomorphic functions associated to the sum of squares domain $D.$
Recall that for any given ideal of holomorphic functions $\left< f \right> = \left< f_1, \dots, f_N \right>$  associated to $D,$  $\text{J(f)} \text{J}^{*}(f)$ is the Levi matrix of $M,$ where $\text{J}(f)$ and $\text{J}^{*}(f)$ are the complex Jacobian matrix and its conjugate transpose respectively. Also, recall that for every vector field $L_j \in T^{1,0}M$ we associate the row vector $v_j,$ $1 \leq j \leq n,$ where $v_1 = (1,0, \dots, 0).$ The reader should note that if any vector field
is denoted by $L^j,$ then its corresponding row vector will also be denoted by $v^j.$ From \eqref{normaldirection}, the $1 \times n $ matrix $\left[ \partial r_1\right]$ is given by $\left[ \partial r_1\right] :=
\begin{pmatrix}
\partial_{z_1}r_1 & \partial_{z_2}r_1 & \cdots & \partial_{z_n}r_1
\end{pmatrix}.$ In general, we will denote by $\left[ \partial r_k\right]$ the matrix 
\[
\left[ \partial r_k\right]:=\begin{pmatrix}
\partial_{z_1}r_k & \partial_{z_2}r_k & \cdots & \partial_{z_n}r_k
\end{pmatrix},
\]
for some $k \in \Z^{+},$ $r_k$ is some real-valued function and $\partial_{z_i}:= \frac{\partial}{\partial z_i}.$

\subsection{The Commutator Multitype Machinery: Ideal Version} 

Let $D$ be a sum of squares domain and let $z_0$ be a point in $M.$ We assume here that the entries of both the multitype and the commutator multitype are not necessarily finite. Let $r$ be a smooth boundary-defining function and suppose that the rank of the Levi form of $M$ at $z_0$ equals $q.$ Let $\left< f\right> = \left< f_1, \dots, f_N \right>$ be an ideal of holomorphic functions associated to $D$ and let $\mathfrak{C}(z) = (c_1, \dots, c_n)$ be the commutator multitype. Set $r =r_1.$

The number $c_1$ is defined by $c_1=1,$ which corresponds to the weight in the $z_1$-direction. Here, $c_1$ is obtained from the condition $\left[ \partial r_1 \right] v_1^{T},$ where $v_1=(1,0, \dots, 0)$ is the vector associated to the vector field $L_1 = \frac{\partial}{\partial z_1} \in T^{1,0}M.$ Since the rank of the Levi form of $M$ at $z_0$ is $q$ we get $c_2= \cdots = c_{q+1}=2.$ Hence, we choose holomorphic vector fields $L_2, \dots, L_{q+1}$ such that $v_j \in \ker \left[ \partial r_1 \right] $ and such that the $q \times q$ matrix
\begin{equation} \label{nonsing}
V \text{J}(f)\text{J}^{*}(f) V^{*}(z_0) \end{equation} 
is non-singular, where $V=[v_{kj}],$ $2 \leq k \leq q+1, \ 2 \leq j \leq n,$ is a $q \times (n-1)$ matrix whose $l$th row is $v_l= (v_{l2} \cdots v_{ln}).$ We denote by $\mathscr{L}_{\mathrm{h}_{\text{max}}}$  the set of the vector fields $L_2, \dots, L_{q+1},$ which is a maximal collection satisfying the condition in \eqref{nonsing}. The construction is done if $q+1=n;$ otherwise, we continue the procedure in the following two steps:

\medspace

\begin{itemize}
	\item[\bf{STEP 1:}] Let $T^{1,0}_{q+2}$ denote the subbundle of $T^{1,0} M$ defined by
	\[
	T^{1,0}_{q+2} = \{ L \ | \ v \in \ker \left[ \partial r_1 \right] \ \text{and} \  v \text{J}(f)\text{J}^{*}(f) v_j^{*}=0, \ j=2 , \dots, q+1 \},
	\]
	where $v$ is the row vector associated to $L \in T^{1,0}M.$ At this point, we recall Definition \ref{equalorder} and Remark \ref{equalorderrk} given in section 4. Let $s$ denote an integer with $s \geq 2,$  $\mathscr{L} = \{L^1, \dots, L^{l}\}$ an equally-ordered list, and  $\mathscr{L}_h = \{L^1, \dots, L^{s}\}$  its associated holomorphic list. Let $l_k$ be the number of times both vector fields $L_k$ and $\bar{L}_k$ appear in $\mathscr{L}$ and let $s_k$ be the number of times that $L_k$ appear in $\mathscr{L}_h.$ Then $l_k=2s_k.$
	
	Now, suppose that $\mathscr{L}_{h} =\{L^1, \dots, L^s\}$ is a list of holomorphic vector fields such that there exists a nonvanishing vector field $L \in T^{1,0}_{q+2}$ such that $L^j=L$ for every $j,$ $1 \leq j \leq s$ and let $\mathscr{L}^{'}_{h}=\{L^2, \dots, L^s\}.$ Denote by $\mathfrak{F}(z)$ the function
	\[
	\mathfrak{F}(z) = v \text{J}(\mathscr{L}^{'}_hf) \text{J}^{*}(\mathscr{L}^{'}_hf) v^{*} (z)
	\]
	for $z \in M,$ where $\text{J}^{*}(\mathscr{L}_{h}^{'}f)$ and $v^{*}$ are the conjugate transposes of $\text{J}(\mathscr{L}_{h}^{'}f)$ and  $v$ respectively. If $\mathfrak{F}(z_0) = 0$ for all such lists, then we set $c_i = \infty$ for $i = q+2, \ldots, n.$ Let $\mathscr{L}_{h} = \{L^1, \dots, L^s\}$ be a list  of holomorphic vector fields that satisfies the condition that $\mathfrak{F}(z_0) \neq 0$ and  denote by $\mathfrak{h}_{\tau}(z)$ the function
	\begin{equation} \label{mmm}
	\mathfrak{h}_{\tau}(z) = v^2 \text{J}(\mathscr{L}^{''}_hf) \text{J}^{*}(\mathscr{L}^{'}_hf) (v^1)^{*} (z),
	\end{equation}
	where $v^j$ is the row vector associated to the vector field $L^j \in \mathscr{L}_{h}(\tau).$ If $\mathfrak{F}(z_0) \neq 0$ for some list $\mathscr{L}_{h},$ then choose one denoted by $\mathscr{L}_{h}(q+2)$ for which $s$ is the smallest and set $c_{q+2} = 2s,$ where $\mathscr{L}_{h}(q+2)= \{L^1, \dots, L^s\}.$ Let $\mathscr{L}^{'}_{h}(q+2)= \{L^2, \dots, L^s\}$ and set $r_{q+2} = \mathrm{Re}\,\mathfrak{h}_{q+2}$ or $r_{q+2} = \mathrm{Im}\,\mathfrak{h}_{q+2}.$ We also denote by $L_{q+2}$ the particular vector field $L$ used to obtain $\mathscr{L}_{h}(q+2).$  This ensures that the condition $L_{q+2}r_{q+2} \neq 0$ holds.
	
	Next, if $q+2=n,$ then the construction terminates; otherwise, we find the subsequent entries of the commutator multitype by following a slightly modified version of the first step. 
	
	\medspace
	
	\item[\bf{STEP \large 2:}] Assume that for $\nu-1 \in \Z,$ $q+2 \leq \nu-1 < n,$ the positive rational numbers $c_1, \ldots, c_{\nu},$ vector fields $L_2, \ldots, L_{\nu},$ and functions $r_j,$ $j=1, q+2, \ldots, \nu,$ have been constructed.  Now, we define a $\nu$-admissible list as follows:
	
	\begin{definition} \label{admseol}
		Let $\mathscr{L}_h = \{L^1, \dots, L^{s}\}$ be an ordered list of holomorphic vector fields. We say that $\mathscr{L}_h$	is $\nu$-admissible if $l_{\nu}>0$ and $\mathscr{L}_h$ satisfies the inequality
		\[
		\sum_{k=q+2}^{\nu-1} \frac{s_k}{c_k} <\frac{1}{2}.
		\]
	\end{definition}
	
	\medspace
	
	\noindent We proceed with the argument by affirming that the following properties also hold from the inductive hypothesis: For every $j,$ $q+2 \leq j \leq \nu -1$
	
	\smallskip

	\begin{itemize}
		\item[(i).] The list $\mathscr{L}_h(j) = \{L^1, \ldots, L^s\}$ is ordered and $j$-admissible.
		
		\item[(ii).] If $\mathscr{L}_h(j) = \{L^1, \ldots, L^s\}$ is an ordered list, $s^j_k$ equals the number of times $L_k$ occur in $\mathscr{L}_h(j),$  and $\displaystyle \sum_{k=q+2}^{j-1} \frac{s_k^j}{c_k} <\frac{1}{2},$ then $\mathfrak{F}(z_0) =0.$ Hence, $\mathscr{L}_h(j)$ is of minimal length.
		
		\item[(iii).] $\mathscr{L}_h(j)(z_0) \neq 0$ 
		
		\item[(iv).] If $\mathscr{L}_h(j) = \{L^1, \ldots, L^s\},$ then the condition $L_j r_j \neq 0$ holds where $\mathscr{L}^{'}_h(j) = \{L^2 , \ldots, L^s\}$ and $r_j = \text{Re}\,\mathfrak{h}_j$ or $r_j = \text{Im}\,\mathfrak{h}_j$ (we are using the notation in \eqref{mmm}). 
		
		\item[(v).] $L_j r_k =0$ if $q+2 \leq k < j \leq \nu.$

		\item[(vi).] If $s_k^j$ equals the number of times $L_k$ occur in $\mathscr{L}_h(j),$ then $s_k^j = 0$ whenever $k > j$ and $\displaystyle \sum_{k=q+2}^{j-1} \frac{s_k^j}{c_k} =\frac{1}{2}.$
		
	\end{itemize}
	
	\medspace
	
	\noindent	Next, denote by $T^{1,0}_{\nu}$ the subbundle of $T^{1,0} b\Omega$ defined by
	\[
	T^{1,0}_{\nu} = \{ L \ | \ v \in \ker \left[ \partial r_k \right] \ \text{and} \  v \text{J}(f)\text{J}^{*}(f) v_j^{*}=0, \ j =2 , \dots, q+1 \},
	\]
	for $ k=1, q+2, \dots, \nu-1.$ We compute the positive rational number $c_{\nu}$ by considering $j=\nu$ in properties (i) and (iii)-(v) above. We consider all ordered $\nu$-admissible list $\mathscr{L}_h$ for every holomorphic vector field $L_{\nu} \in T^{1,0}_{\nu}.$ If $\mathfrak{F}(z_0) =0$ for all such lists $\mathscr{L}_h,$ then we set $c_{\nu} = \cdots = c_n = \infty.$ If $\mathfrak{F}(z_0) \neq 0,$ then we choose one with minimal length and denote it by $\mathscr{L}_h(\nu),$ where $\mathscr{L}_h(\nu) = \{L^1, \dots, L^s\}.$ Let $s_k$ be the number of times that $L_k$ appear in the list $\mathscr{L}_h(\nu)$ and let $c(\mathscr{L})$ denote the solution to
	\begin{equation} 
	\sum_{q+2}^{\nu-1} \frac{s_k}{c_k} + \frac{s_{\nu}}{c(\mathscr{L})} =\frac{1}{2}.
	\end{equation}
	Since $\mathscr{L}_h(j)$ is $\nu$-admissible, the solution $c(\mathscr{L}) \in \Q^{+}.$ Let $\mathscr{L}_h^{'}(\nu) = \{L^2, \dots , L^s\}.$ Set $c_{\nu} = c(\mathscr{L}_h(\nu))$ and $r_{\nu} = \text{Re}\,\mathfrak{h}_{\nu}$ or $r_{\nu} = \text{Im}\,\mathfrak{h}_{\nu}$ so that  $L_{\nu} r_{\nu} \neq 0$ is satisfied. Also, the list $\mathscr{L}_{\nu}$ satisfies properties (i)-(vi). The process terminates after the last entry is obtained. Thus, this procedure generates  $\mathfrak{C}(z_0) = (c_1, \dots, c_n)$ the commutator multitype at $z_0.$ Let $\nu \leq n$ be the largest index such that $c_{\nu} < \infty.$ We call the collection 
	\[
	\mathfrak{B}_{\nu}(z_0) = \{r_1, r_{q+2}, \ldots, r_{\nu}; \ L_2, \ldots, L_{\nu} \}
	\]
	of real-valued functions and vector fields  a \textit{boundary system} of rank $q$ and codimension $n - \nu.$
\end{itemize}

\subsection{The Model Domain Case.} We shall consider the Catlin's machinery for the case where the domain in question is $D_0,$ the model of a sum of squares domain $D.$ Assume that all the entries of both the multitype and the commutator multitype are finite. Let $M_0 \subset \C^n \ n \geq 2$ be a model hypersurface of a model sum of squares domain $D_0$  defined by $\{r_0 < 0\},$ where
\[
r_0 = 2\mathrm{Re} z_1 + \sum_{l=1}^N|w_l(z_2, \dots, z_n)|^2,
\]
and $w_l$ is a polynomial of weight $1/2$ with respect to the Catlin multitype at $0.$ Assume that the rank of the Levi form of $M_0$ is $q$ and let $z_0$ be a point in $M_0.$ Here, we will adopt a slightly different approach to obtaining the holomorphic vector fields in the maximal collection $\mathscr{L}_{h_{\text{max}}}.$ We will derive holomorphic functions $\mathfrak{h}_2, \ldots, \mathfrak{h}_{q+1}$ whose mode of construction will be consistent with methods used to obtain the functions $\mathfrak{h}_k$ for $k = q+2, \ldots, n$ in ideal reformulation of the Catlin commutator multitype machinery in subsection 6.1 above. Let $\left< w\right> = \left<w_1, \dots, w_N \right>$ be the ideal of holomorphic functions associated to the domain $D_0.$ The holomorphic functions $\mathfrak{h}_2, \ldots, \mathfrak{h}_{q+1}$ are constructed as follows: We denote by $W_t(z)$ the function 
\[
W_t(z) = v_t \text{J}(w) \text{J}^{*}(w)v_t^{*}(z),
\]
where $v_t$ is the row vector associated to the (1,0) vector field $L_t,$ for some $t \in \Z_{> 0},$ and $\text{J}^{*}(w)$ is the conjugate transpose of $\text{J}(w).$ Since the rank of the Levi form is $q,$ we can always find (1,0) vector fields $L_j \in \mathscr{L}_{h_{\text{max}}},$ $j = 2, \dots, q+1,$ such that the function $\mathfrak{h}_j(z)$ given by 
\[
\mathfrak{h}_j(z) = \begin{pmatrix} w_1 & \cdots & w_N \end{pmatrix} \text{J}^{*}(w) v_j^{*}(z),
\]
where $v_j$ is the row vector associated to the vector field $L_j,$ is holomorphic. Also, each $L_j$ satisfies the conditions that $W_j(z_0) \neq 0$ and $L_j \in \ker \left[ \partial r_1\right].$  The condition $W_j(z_0) \neq 0$ can be interpreted for each vector field $L_j \in \mathscr{L}_{h_{\text{max}}},$ $2 \leq j \leq q+1,$ as $v_j \text{J}(\mathfrak{h}_j)(z_0) \neq 0.$ It is also easy to see that we can always choose the vector fields $L_j$ and $L_k$ for $2 \leq j< k \leq q+1,$ so that the condition $v_k\text{J}(\mathfrak{h}_j)(z_0) = 0$ holds. Denote by $\mathscr{W}(z)$ the function
\[
\mathscr{W}(z) = v \text{J}(\mathscr{L}^{'}_hw) \text{J}^{*}(\mathscr{L}^{'}_hw) v^{*} (z)
\]
for $z \in M,$ where $\text{J}^{*}(\mathscr{L}_{h}^{'}w)$ and $v^{*}$ are the conjugate transposes of $\text{J}(\mathscr{L}_{h}^{'}w)$ and  $v$ respectively.
By applying Corollary \ref{corinbnd} and Lemma \ref{rshol} to $r_0,$ we can always find a list $\mathscr{L}_{h}(\tau) = \{L^1, \dots, L^s\},$ $q+2 \leq \tau \leq n,$ of holomorphic vector fields such that $\mathscr{W}(z_0) \neq 0$ holds and such that the function
\[
\mathfrak{h}_{\tau}(z) = v^2 \text{J}(\mathscr{L}^{''}_hw) \text{J}^{*}(\mathscr{L}^{'}_hw) (v^1)^{*} (z)
\]
is holomorphic. Here $v^j$ is the row vector associated to the vector field $L^j \in \mathscr{L}_{h}(\tau).$ Continuing in a similar manner as in the construction given in section 6.1, we  obtain the following: For every minimal list $\mathscr{L}_{h}(\tau) = \{L^1, \dots, L^s\},$ $q+2 \leq \tau \leq n,$  such that $\mathscr{W}(z_0) \neq 0$ holds, we get
\[
c_{\tau} = c(\mathscr{L}_h(\tau)) \quad \text{and} \quad r_{\tau} = \text{Re}\,\mathfrak{h}_{\tau} \ \ \text{or} \ \ r_{\tau} = \text{Im}\,\mathfrak{h}_{\tau} 
\]
so that  $L_{\tau} r_{\tau} \neq 0$ is satisfied. Since $\mathfrak{h}_{\tau}(z)$ is holomorphic, the condition that $L_{\tau} r_{\tau} \neq 0$ is equivalent to the condition 
\begin{equation} \label{fin}
v_{\tau} \text{J}(\mathfrak{h}_{\tau})(z_0) \neq 0,
\end{equation}
where $\text{J}(\mathfrak{h}_{\tau})$ is the complex Jacobian matrix of $\mathfrak{h}_{\tau}$ and $v_{\tau}$ is the row vector associated to the vector field $L_{\tau}.$ Also, $v_{j} \text{J}(\mathfrak{h}_{\tau})(z_0) = 0,$ if $j > \tau.$ Finally, we shall consider an ideal generated by all the constructed holomorphic functions. Now, let $\nu \leq n$ be the highest index such that $c_{\nu} < \infty.$ Then for any boundary system $\mathfrak{B}_{\nu},$ we denote by  $\mathcal{I}_{\mathfrak{B}_{\nu}}$ the ideal 
\[
\mathfrak{h} = \left< \mathfrak{h}_2, \dots, \mathfrak{h}_{\nu}\right>
\]
generated by the holomorphic functions $\mathfrak{h}_j,$ $j=2, \dots, \nu.$ We shall call $\mathcal{I}_{\mathfrak{B}_{\nu}}$ \textit{the boundary-system ideal.}

\bibliography{Article2}
\bibliographystyle{alpha}

\end{document}